\documentclass[11pt,reqno,a4paper]{amsart}

\usepackage{amssymb,epsfig,graphics}
\usepackage{amsmath}

\newtheorem{theorem}{Theorem}[section]

\newtheorem{corollary}[theorem]{Corollary}
\newtheorem{lemma}[theorem]{Lemma}
\newtheorem{sub-lemma}[theorem]{Sub-Lemma}

\newtheorem{remark}[theorem]{Remark}

\def\L{\mathcal{L}}

\def\RR{\mathbb{R}}

\let\eps=\varepsilon

\def\B{\mathcal{B}}

\def\RR{{\mathbb R}}
\def\1{{{\mathit 1} \!\!\>\!\! I} }

\newcommand{\pa}{\partial_{\alpha}}

\begin{document}

\title{Linear response in the intermittent family: differentiation in a weighted $C^0$-norm}
\author{Wael Bahsoun}
\address{Department of Mathematical Sciences, Loughborough University,
Loughborough, Leicestershire, LE11 3TU, UK}
\email{W.Bahsoun@lboro.ac.uk}
\author{Beno\^\i t Saussol}
\address{Universit\'e de Brest, Laboratoire de
Math\'ematiques de Bretagne Atlantique, CNRS UMR 6205, Brest, France}
\email{benoit.saussol@univ-brest.fr}
\thanks{This work was conducted during mutual visits of WB to Universit\'e de Bretagne Occidentale and of BS to Loughborough University. WB and BS would like to thank The Leverhulme Trust for supporting their research visits through the Network Grant IN-2014-021. 
}
\begin{abstract}
We provide a general framework to study differentiability of SRB measures for one dimensional non-uniformly expanding maps. Our technique is based on inducing the non-uniformly expanding system to a uniformly expanding one, and on showing how the linear response formula of the non-uniformly expanding system is inherited from the linear response formula of the induced one. We apply this general technique to interval maps with a neutral fixed point (Pomeau-Manneville maps) to prove differentiability of the corresponding SRB measure. Our work covers systems that admit a finite SRB measure and it also covers systems that admit an infinite SRB measure. In particular, we obtain a linear response formula for both finite and infinite SRB measures. To the best of our knowledge, this is the first work that contains a linear response result for infinite measure preserving systems.
\end{abstract}
\date{\today}
\maketitle
\markboth{Wael Bahsoun \and Beno\^\i t Saussol}{Linear response in the intermittent family}
\bibliographystyle{plain}
\tableofcontents
\section{Introduction}
In physical applications of dynamical systems, it is important to understand how statistical properties of a perturbed physical system are related to statistical properties of the original system; i.e., before the occurrence of the perturbation. In particular, it is always desirable to write a first order approximation of the Sinai-Ruelle-Bowen (SRB) measure of the perturbed system in terms of the SRB measure of the original system. In smooth ergodic theory, this direction of research, which was pioneered by David Ruelle, is called differentiation (with respect to noise) of SRB measures. In the physics literature the equivalent term is called `linear response'.  

\bigskip

Linear response has been proved for several classes of smooth dynamical systems that admit exponential, or at least summable, decay of correlations \cite{Ba1, BS, BL, D, GL, KP, R}. Negative results, where linear response does not hold, are also known \cite{Ba1, Ba2}. A recent survey on the progress in this area of research is \cite{Ba2}. More recently, results on the linear response of polynomially mixing systems that admit a probabilistic SRB measure were announced in \cite{BT,K}. Such systems have attracted the attention of both mathematicians \cite{ LSV, young99} and physicists because of their importance in the study of intermittent transition to turbulence \cite{PM}. 

\bigskip 

In this work we provide a general framework to study differentiability of SRB measures for one dimensional non-uniformly expanding maps. We use this general framework to study linear response of maps with neutral fixed points. In particular, we apply our results to study linear response of Pomeau-Manneville type maps \cite{LSV, PM}. The difference between our result and those of \cite{BT, K} is two-fold: in \cite{BT,K} the authors obtain results only for probabilistic SRB measures. Moreover, they obtain a weak form of differentiability. While in our work, we cover both the finite and infinite SRB measure cases and we prove differentiability in norm\footnote{Theorem 1.2 of Korepanov \cite{K} implies differentiability in norm for the LSV map but only for probabilistic SRB measures. See the discussion on page 2 of \cite{K}. We would also like to stress here that Theorem 1.2 of \cite{K} uses the explicit formula of LSV maps and it does not cover the infinite SRB measure case.}. Moreover, we provide a linear response formula that covers both the finite and infinite SRB measure cases.

\bigskip

In Section \ref{setup} we introduce a general setup for the systems we study and we state our assumptions on this general setup. Section \ref{msection} includes the statement of our main results (Theorems~\ref{thm:A} and~\ref{thm:B}). Section \ref{proofs} contains the proof of the theorems  through several lemmas. In Section \ref{LSVmaps} we show that the assumptions of Section \ref{setup} are satisfied by the intermittent maps studied in \cite{LSV}.
\section{Setup and Assumptions}\label{setup}

\subsection{Interval maps with an inducing scheme}\label{ss:T}
We introduce now a class of (family of) interval maps which are non-uniformly expanding with two branches,
for which one can construct an inducing scheme which allow to inherit the linear response formula from 
the one for the induced system.

\begin{itemize}
\item 
Let $V$ be a neighbourhood of $0$. For any $\eps\in V$, $T_\eps\colon [0,1]\to[0,1]$ is a non-singular map, with respect to Lebesgue measure, $m$, with two onto branches $T_{0,\eps}\colon[0,1/2]\to[0,1]$ and $T_{1,\eps}\colon[1/2,1]\to[0,1]$. The inverse branches of  $T_{0,\eps}$, $T_{1,\eps}$ are respectively denoted by $g_{0,\eps}$ and $g_{1,\eps}$. We call $T_0:=T$ the unperturbed map, and $T_\eps$, for $\eps\not= 0$, the perturbed map.
\item We assume that for each $i=0,1$ and $j=0,1,2$ the following partial derivatives exist and satisfy the commutation relation
\begin{equation}\label{comm}
\partial_\eps g_{i,\eps}^{(j)} = (\partial_\eps g_{i,\eps})^{(j)}.
\end{equation}
\item We assume that $T_\eps$ has a unique absolutely continuous invariant measure\footnote{The $T_\eps$ absolutely continuous invariant measure is not assumed to be probabilistic; we allow for $T_\eps$ to admit a $\sigma$-finite absolutely continuous invariant measure.} (up to multiplication) whose Radom-Nykodim derivative will be denoted by $h_\eps$, and we denote for simplicity $h=h_0$.
\item Let $\hat T_\eps$, be the first return map of $T_\eps$ to $\Delta$, where $\Delta:= [1/2,1]$; i.e., for $x\in\Delta$
$$\hat T_{\eps}(x)=T_\eps^{R_\eps(x)}(x),$$
where 
$$R_\eps(x)=\inf\{n\ge 1:\, T^{n}_\eps(x)\in\Delta\}.$$
We assume that $\hat T_\eps$ has a unique acim (up to multiplication) with a continuous density denoted $\hat h_\eps\in C^0$.
\item Let $\Omega$  be the set of finite sequences of the form $\omega = 10^n$, for $n \in\mathbb N \cup\{0\}$. We set $g_{\omega,\eps}=g_{1,\eps}\circ g_{0,\eps}^{n}$. Then for $x\in [0,1]$ we have $T_{\eps}^{n+1}\circ g_{\omega,\eps}(x)=x$. The cylinder sets $[\omega]_{\eps} = g_{\omega, \eps}(\Delta)$, form a partition of $\Delta$ (mod $0$).
For $x\in [0,1]$, we assume
\begin{equation}\label{a4}
\sup_{\eps\in V}\sup_{x\in[0,1]}|g'_{\omega,\eps}(x)| <\infty ;
\end{equation}
\begin{equation}\label{a4.0}
\sup_{\eps\in V}\sup_{x\in[0,1]}|\partial_\eps g_{\omega,\eps}(x)| < \infty;
\end{equation}
\begin{equation}\label{a4'}
\sum_{\omega}\sup_{\eps\in V}||g'_{\omega,\eps}||_{\mathcal B} <\infty;
\end{equation}
and
\begin{equation}\label{a5}
\sum_\omega\sup_{\eps\in V}||\partial_\eps g_{\omega,\eps}'||_{\mathcal B}<\infty,
\end{equation}
\end{itemize}
where $\mathcal{B}$ denotes the set of continuous functions on $(0,1]$ with the norm $$\parallel f\parallel_{\mathcal{B}}=\sup\limits_{x\in(0,1]}|x^{\gamma}f(x)|,$$ for a fixed\footnote{In \eqref{a4'} and \eqref{a5} we need the assumptions to hold
only for a single $\gamma$.} $\gamma>0$. When equipped with the norm $\parallel \cdot\parallel_{\mathcal{B}}$, $\mathcal{B}$ is a Banach space. 

For $\Phi\in L^1$, let 
\begin{equation}\label{theF}
F_\eps(\Phi):=1_\Delta\Phi + (1-1_\Delta)\sum_{\omega\in\Omega}\Phi\circ g_{\omega,\eps}g_{\omega,\eps}'.
\end{equation}
Note that $F_\eps$ is a linear operator. In fact, for $x\in[0,1]\setminus \Delta$, the formula of $F_{\eps}$ can be re-written using the Perron-Frobenius operator of $T_\eps$: 
$$F_\eps(\Phi):=1_\Delta\Phi + (1-1_\Delta)\sum_{k\ge1}L^{k}_{\eps}(\Phi\cdot 1_{\{R_\eps>k\}}),$$
where $L_\eps$ is the Perron-Frobenius operator associated with $T_\eps$; i.e., for $\varphi\in L^{\infty}$ and $\psi\in L^1$
$$\int\varphi\circ T_\eps\cdot\psi dm=\int \varphi\cdot L_\eps\psi dm.$$
It is well known, see for instance \cite{BV}, that the densities of the original system and the induced one are related (modulo normalization in the finite measure case) by
\begin{equation}\label{eq:density}
h_\eps = F_\eps(\hat h_\eps).
\end{equation}

We also define the following operator, which will represent $\partial_\eps F_\eps\Phi|_{\eps=0}$ 
\begin{equation}\label{eq:Q}
Q\Phi = (1-1_\Delta) \sum_\omega \Phi'\circ g_\omega\cdot a_\omega g_\omega' + \Phi\circ g_\omega\cdot b_\omega,
\end{equation}
where $a_\omega=\partial_\eps g_{\omega,\eps}|_{\eps=0}$ and $b_\omega=\partial_\eps g_{\omega,\eps}'|_{\eps=0}$.

\subsection{Interval maps with countable number of branches}\label{ss:hatT}
We introduce here a class of (family of) interval maps which are uniformly expanding,
with a finite or countable number of branches, for which we will be able to 
prove a linear response formula. The induced map in Subsection \ref{ss:T} is a particular case of such uniformly expanding maps.\\

Let $\Delta$ be an interval and $V$ be a neighborhood of $0$. Let $\Omega$ be a finite or countable set.
We assume that the maps $\hat T_\eps\colon \Delta\to\Delta$ satisfy
\begin{itemize}
\item
For each $\eps\in V$, there exists a partition (mod 0) of $\Delta$ into open intervals $\Delta_{\omega,\eps}$, $\omega\in\Omega$ such that the restriction of $\hat T_\eps$ to $\Delta_{\omega,\eps}$ is piecewise $C^3$, onto and uniformly expanding in the sense that $\inf_\omega\inf_{\Delta_{\omega,\eps}}|\hat T_{\omega,\eps}'|>1$. We denote by $g_{\omega,\eps}$ the inverse branches of $\hat T_\eps$ on $\Delta_{\omega,\eps}$.
\item
We assume that for each $\omega\in\Omega$ and $j=0,1,2$ the following partial derivatives exist and satisfy the commutation relation\footnote{Note that \eqref{commB} is satisfied when $\hat T_\eps$
is an induced map as in Subsection~\ref{ss:T}. In particular, for each $i=0,1$ and $j=0,1,2$ the following partial derivatives exist and satisfy the commutation relation $\partial_\eps g_{i,\eps}^{(j)} = (\partial_\eps g_{i,\eps})^{(j)}$.}
\begin{equation}\label{commB}
\partial_\eps g_{\omega,\eps}^{(j)} = (\partial_\eps g_{\omega,\eps})^{(j)}.
\end{equation}
\item
We assume
\begin{equation}\label{a1}
\sum_{\omega}\sup_{\eps\in V}\sup_{x\in\Delta}|g'_{\omega, \eps}(x)|<\infty;
\end{equation}
and 
\begin{equation}\label{a2}
\sup_\omega\sup_{\eps\in V}\sup_{x\in\Delta} \left|\frac{g_{\omega,\eps}''(x)}{g_{\omega,\eps}'(x)}\right|<\infty;
\end{equation}
and for $i=1,2$
\begin{equation}\label{a3}
\sum_\omega\sup_{\eps\in V}\sup_{x\in\Delta}|\partial_\eps g_{\omega,\eps}^{i}(x)| <\infty. 
\end{equation}
\end{itemize}

Let $\hat L_\eps$ denote the Perron-Frobenius operator of the map $\hat T_\eps$; i.e., for  $\Phi\in L^1(\Delta)$ 
$$\hat L_\eps\Phi(x) :=\sum_{\omega\in\Omega}\Phi\circ g_{\omega,\eps}(x)g_{\omega,\eps}'(x)$$
for a.e. $x\in\Delta$. Under these conditions it is well known that $\hat T_\eps$ admits a unique (up to multiplication) invariant absolutely continuous finite measure. We denote its density by $\hat h_\eps$. Hence $\hat L_\eps\hat h_\eps=\hat h_\eps$. Moreover, $\hat L_{\eps}$ has a spectral gap when acting on $C^k$, $k=1,2$ (see for instance \cite{Li}). 
We denote the  Perron-Frobenius operator of the unperturbed map $\hat T$ by $\hat L$; i.e., $\hat L:=\hat L_0$ and let $\hat h:=\hat h_0$.   

\section{Statement of the main results}\label{msection}
\subsection{Statement of the main results}
A first general statement is that the differentiability of the $T_\eps$  absolutely continuous measure is inherited from that of the induced system.
\begin{theorem}\label{thm:A}
Let $T_\eps$ be a family of maps of the interval as described in Subsection~\ref{ss:T}.
If the density $\hat h_\eps$ of the induced map $\hat T_\eps$ is differentiable as a $C^0$ element, that is there exists $\hat h^*\in C^0$ such that 
\begin{equation}\label{eq:diffhat}
\lim_{\eps\to 0}||\frac{\hat h_{\eps}-\hat h}{\eps}-\hat h^*||_{C^0}=0,
\end{equation}
for some $\hat h\in C^0$, then 
\begin{enumerate}
\item there exists $h^*\in\B$ such that
$$\lim_{\eps\to 0}||\frac{h_{\eps}-h}{\eps}-h^*||_{\mathcal B}=0;$$
i.e., $h_\eps$ is differentiable as an element of $\mathcal B$ with respect to $\eps$;
\item in particular, if the conditions hold for some $\gamma<1$
$$\lim_{\eps\to 0}||\frac{h_{\eps}-h}{\eps}-h^*||_{1}=0.$$
\item The function $h^*$ is given by 
\footnote{Note that in the finite measure case, $h^*$ is the derivative of the non-normalized density $h_\eps$. The advantage in working with $h_{\eps}$ is reflected in keeping the operator $F_{\eps}$ linear and to accommodate the infinite measure preserving case. In the finite measure case, once the derivative of $h_{\eps}$ is obtained, the derivative of the normalized density can be easily computed. Indeed, $h_\eps=h +\eps h^*+o(\eps)$. Consequently, $\int h_\eps=\int h +\eps \int h^*+o(\eps)$. Hence, $\partial_{\eps}(\frac{h_\eps}{\int h_\eps}){|}_{\eps=0}=h^*-h\int h^*$.}
\[
h^* = F_0( \hat h^* ) + Q\hat h.
\]
\end{enumerate}
\end{theorem}

Next, we show that for the family of maps with countable number of branches introduced in Subsection~\ref{ss:hatT} the invariant density is differentiable as an element of $C^0$.

\begin{theorem}\label{thm:B}
Let $\hat T_\eps\colon \Delta\to\Delta$ be a family of maps of the interval as described in Subsection~\ref{ss:hatT}.
Then the density $\hat h_\eps$ of the map $\hat T_\eps$ is differentiable as a $C^0$ element, that is there exists $\hat h^*\in C^0$ such that~\eqref{eq:diffhat} holds.
Moreover, we have the linear response formula
$$\hat h^*:=(I-\hat L)^{-1}\hat L[A_0\hat h'+B_0\hat h],$$
where $\hat h'$ is the spatial derivative of $\hat h$ and
$$A_0=-\left(\frac{\partial_{\eps}\hat T_{\eps}}{\hat T'_{\eps}}\right){\Big{|}}_{\eps=0},\hskip 0.5cm B_0= \left(\frac{\partial_{\eps}\hat T_{\eps}\cdot \hat T_{\eps}''}{\hat T_\eps'^2}-\frac{\partial_{\eps}\hat T_{\eps}'}{\hat T_\eps'}\right){\Big{|}}_{\eps=0}.$$ 
\end{theorem}
\begin{corollary}
If $T_\eps$ satisfies the assumptions of Subsections~\ref{ss:T} and \ref{ss:hatT}, then
\begin{equation}\label{eq:lrf}
h^* = F_0 (I-\hat L)^{-1} \hat L (A_0\hat h'+B_0\hat H) + Q\hat h.
\end{equation}
\end{corollary}
\begin{proof}
The proof follows from Theorems \ref{thm:A} and \ref{thm:B}. 
\end{proof}
\begin{remark}[Moving inducing sets]
We notice that Theorem~\ref{thm:A} generalizes easily to the case where the inducing sets $\Delta_\eps$ are allowed to depend on $\eps$ in a $C^1$ way. Indeed, any $C^1$ family of $C^1$ diffeomorphism $S_\eps\colon [0,1]\to[0,1]$ such that $S_\eps(\Delta_\eps)=\Delta$, $S_0=id$, will conjugate $T_\eps$ with a map $\bar T_\eps$ whose inducing set is $\Delta$. Applying Theorem \ref{thm:A} to the map $\bar T_\eps$, with the obvious notation, we obtain: 
\begin{equation}\label{moving}
\bar h_\eps=\bar h+\eps \bar h^*+o(\eps)
\end{equation} 
Then using \eqref{moving} and the fact that $h_\eps = \bar h_\eps\circ S_\eps\cdot S_\eps'$ we obtain 
\begin{equation}
\partial_\eps h_\eps|_{\eps=0} =  \bar h'\cdot\partial_\eps S_\eps|_{\eps=0}+\bar h^*+\bar h\cdot\partial_\eps S_\eps'|_{\eps=0}.
\end{equation}
\end{remark}

\subsection{Rigorous numerical approximation of the derivative}\label{compare}
An important feature of our approach is that it could be amenable to obtain rigorous numerical approximation of $h^*$. In particular, since $\hat L$ has a spectral gap on $C^k$, $k=1,2$, using ideas of \cite{BGNN} one can approximate $(I-\hat L)^{-1}\hat L[A_0\hat h'+B_0\hat h]$ as a first step, and in the second step one can follow the path of \cite{BBD} and pull back the computed formula of the first step to the full system and obtain a numerical approximation of $h^*$ in $\mathcal B$.

\section{Proof of the results}\label{proofs}
We use the letter $C$ to denote positive constants whose values may change when estimating various expressions but are independent of both $\eps$ and $\omega$ (or $n$). In the following, we first present in Subsection \ref{pthb} the proof of Theorem \ref{thm:B}, and then in Subsection \ref{ptha} we present the proof of Theorem \ref{thm:A}. 
\subsection{Proof of Theorem~\ref{thm:B}} \label{pthb}
We first prove a  lemma that will be used in the linear response formula in Theorem \ref{thm:B}. 
\begin{lemma}\label{lem:3b}
For any differentiable function $\Phi\colon\Delta\to\RR$, the function $\Phi\circ g_{\omega,\eps} g_{\omega,\eps}'$ is differentiable with respect to $\eps$
and we have on $\Delta$ 
\begin{equation}\label{eq:2}
\partial_\eps ( \Phi\circ  g_{\omega,\eps} g_{\omega,\eps}') = [ \Phi' A_\eps + \Phi B_\eps]\circ g_{\omega,\eps} g_{\omega,\eps}',
\end{equation}
where
\[
A_\eps=-\left(\frac{\partial_{\eps}\hat T_{\eps}}{\hat T'_{\eps}}\right),\hskip 0.5cm 
B_\eps= \left(\frac{\partial_{\eps}\hat T_{\eps}\cdot \hat T_{\eps}''}{\hat T_\eps'^2}-\frac{\partial_{\eps}\hat T_{\eps}'}{\hat T_\eps'}\right).
\]
\end{lemma}
\begin{proof}
We start from the relation $\hat T_\eps \circ g_{\omega,\eps}(x)=x$ and differentiate it  with respect to $\eps$ and get $\hat T_\eps'\circ g_{\omega,\eps} \partial_\eps g_{\omega,\eps} + \partial_\eps \hat T_\eps \circ g_{\omega,\eps}=0$. This gives $\partial_\eps g_{\omega,\eps} = A_\eps\circ g_{\omega,\eps}$.
This also implies that $\partial_\eps g_{\omega,\eps}' = A_\eps' \circ g_{\omega,\eps} g_{\omega,\eps}' = B_\eps\circ g_{\omega,\eps} g_{\omega,\eps}'$. The conclusion follows from the following differentiation with respect to  $\eps$:
\begin{equation}\label{diffeps}
\begin{split}
\partial_\eps (\Phi\circ g_{\omega,\eps}g_{\omega,\eps}')
&=
\partial_\eps(\Phi\circ g_{\omega,\eps})g_{\omega,\eps}' + \Phi\circ g_{\omega,\eps} \partial_\eps g_{\omega,\eps}'\\
&=
\Phi'\circ g_{\omega,\eps} \partial_\eps g_{\omega,\eps}g_{\omega,\eps}' + \Phi\circ g_{\omega,\eps} \partial_\eps g_{\omega,\eps}'.
\end{split}
\end{equation}
\end{proof}
\noindent{\bf Strategy of the proof of Theorem \ref{thm:B}.} The general strategy starts from the identity

\begin{lemma}\label{lem:1}
We have $\hat h_\eps = (I-\hat L_\eps)^{-1}(\hat L_\eps-\hat L)\hat h + \hat h$.
\end{lemma}
\begin{proof}
One easily checks that 
 \[
(I-\hat L_\eps)( \hat h_\eps-\hat h) = (\hat L_\eps-\hat L) \hat h.
 \]
Since $\hat L_\eps$ has a spectral gap on $C^1$ it eventually contracts exponentially
 on the subset of zero average functions $C_0^1$. Since the ranges of 
 $(\hat L_\eps-\hat L_0)$ and $(I-\hat L_\eps)$ are contained in $C_0^1$, the composition below is well defined
 $$(I-\hat L_\eps)^{-1}(I-\hat L_\eps)( \hat h_\eps-\hat h) = (I-\hat L_\eps)^{-1}(\hat L_\eps-\hat L) \hat h.$$
 This completes the proof of the lemma.
 \end{proof}

Setting $H_\eps = \hat L_\eps -\hat L$ and $G_\eps=(I-\hat L_\eps)^{-1}$, Lemma~\ref{lem:1} reads 
\begin{equation}\label{eq:GH}
\hat h_\eps = G_\eps H_\eps \hat h +\hat h.
\end{equation}
We then obtain, using Lemma~\ref{diff-lemma} below, the following first order expansion in $C_0^1$
\[
H_\eps \hat h = \eps q + o(\eps).
\]
We then show, see second statement of Lemma~\ref{lem:G} below, that $G_\eps$ is uniformly bounded in $\L(C_0^1,C^0)$ to obtain the following expansion in $C^0$
\[
G_\eps H_\eps\hat h = \eps G_\eps q + o(\eps).
\]
Finally, using the two expansions above with \eqref{eq:GH} and showing that $G_\eps(q)\to G_0(q)$ in $C^0$, see the first statement of Lemma~\ref{lem:G} below, we obtain in $C^0$
\[
\hat h_\eps = \hat h + \eps G_0(q) + o(\eps),
\]
which proves the theorem.
\begin{lemma}\label{diff-lemma}
We have 
\[\frac{H_\eps \hat h}{\eps} \to q \text{ in } C_0^1,
\]
where $q=\hat L [A_0\hat h' + B_0 \hat h]$.
\end{lemma}
\begin{proof}
Recall that $H_\eps=\hat L_\eps-\hat L$ hence we need to show that $\eps\mapsto\hat L_\eps\hat h$ is differentiable as a $C^1$ element, on some neighborhood $V$ of $0$. To this end, recall that $\hat L_\eps\hat h = \sum_\omega\hat h \circ g_{\omega,\eps} g_{\omega,\eps}'$.
It suffices to show that

(i) for each $\omega$, the map $\eps\in V\mapsto \hat h\circ g_{\omega,\eps} g_{\omega,\eps}'\in C^1$ is differentiable;

(ii) the series $\sum_\omega \sup_{\eps\in V} \|\partial_\eps(\hat h\circ g_{\omega,\eps} g_{\omega,\eps}')\|_{C^1}<\infty$.

We first prove (i).
Drop for simplicity the subscript $\omega$ and write $g_\eps=g_{\omega,\eps}$ and let $f_\eps=\hat h\circ g_\eps g_\eps'$. We have
\[
\begin{split}
f_\eps &= \hat h\circ g_\eps g_\eps'\\
f_\eps' &= \hat h'\circ g_\eps (g_\eps' )^2+ \hat h \circ g_\eps g_\eps''.
\end{split}
\]
By the commutation relations given by assumption \eqref{commB} we have
\begin{equation}\label{eq:com}
\partial_\eps f_\eps^{(i)} = (\partial_\eps f_\eps)^{(i)},\quad i=0,1
\end{equation}
and these are continuous functions on $\Delta\times V$.\\

Let $\nu\in V$ and $\eps$ be small. We have
\begin{equation}\label{eq:W2norm}
\|f_{\eps+\nu}-f_\nu-\eps(\partial_\delta f_\delta|_{\delta=\nu})\|_{C^1}
=
\sum_{i=0}^1
\|f_{\eps+\nu}^{(i)}-f_\nu^{(i)}-\eps(\partial_\delta f_\delta|_{\delta=\nu})^{(i)}\|_{C^0}.
\end{equation}
For each $x$, by the mean value theorem, there exists $\eta_{x,\eps}^i$ such that
$f_{\eps+\nu}^{(i)}(x)-f_\nu^{(i)}(x)=\eps \partial_\delta f_\delta^{(i)}|_{\delta=\eta_{x,\eps}^i}$, with $|\eta_{x,\eps}^i-\nu|<\eps$. Therefore
\[
\sum_{i=0}^1
\|f_{\eps+\nu}^{(i)}-f_\nu^{(i)}-\eps(\partial_\delta f_\delta^{(i)}|_{\delta=\nu})\|_{C^0}
\le |\eps| \sum_{i=0}^1 \ \|\partial_\delta f_\delta^{(i)}|_{\delta={\eta_{\cdot,\eps}^i}}-\partial_\delta f_\delta^{(i)}|_{\delta=\nu}\|_{C^0}
=o(\eps).
\]
We conclude by \eqref{eq:W2norm} and the commutation relation \eqref{eq:com}. We now prove (ii).
\begin{equation}\label{eq:sum}
\sum_\omega\sup_{\eps\in V}\|\partial_\eps f_{\omega,\eps}\|_{C^1}
=
\sum_\omega\sup_{\eps\in V}\sum_{i=0}^1 \| \partial_\eps f_{\omega,\eps}^{(i)}\|_{C^0}.
\end{equation}
We write for $i=0,1$ 
\begin{equation}\label{eq:ww}
\partial_\eps f_{\omega,\eps}^{(i)} = 
\sum_{k=0}^{i+1} a_k^{(i)} \partial_\eps g_{\omega,\eps}^{(k)},
\end{equation}
where the coefficients $a_k^{(i)}$ are given respectively by
\[
a_0^{(0)} = \hat h'\circ g_{\omega,\eps} g_{\omega,\eps}'
,\quad
a_1^{(0)} = \hat h\circ g_{\omega,\eps}
\]
then differentiating again in space we get
\[
a_0^{(1)} = \hat h''\circ g_{\omega,\eps}g_{\omega,\eps}'^2+\hat h'\circ g_{\omega,\eps}g_{\omega,\eps}'',
\quad
a_1^{(1)} = 2\hat h'\circ g_{\omega,\eps}g_{\omega,\eps}',\quad
a_2^{(1)} = \hat h\circ g_{\omega,\eps}.
\]
By assumptions \eqref{a2} and \eqref{a4}, we have
 $a_0^{(i)} \le C g_{\omega,\eps }'$ and $a_k^{(i)}\le C$ for any $i=0,1$ and $k\neq0$. Moreover, by assumption \eqref{a3}, for $k=1,2$,
$$
\sum_\omega\sup_{\eps\in V}\sup_{x\in\Delta}|\partial_\eps g_{\omega,\eps}^{(k)}(x)| \le C. 
$$
Putting these estimates together with \eqref{eq:ww} imply that \eqref{eq:sum} is finite, proving (ii). Moreover, we have 
\[
\partial_\eps H_\eps \hat h|_{\eps=0} = \sum_\omega \partial_\eps(\hat h\circ g_{\omega,\eps} g_{\omega,\eps}')|_{\eps=0}=\hat L[A_0\hat h' + B_0 \hat h],
\]
where we have used~\eqref{eq:2}.
\end{proof}
\begin{lemma}\label{lem:0}
For any $\Phi\in C^1$ we have $\hat L_\eps\Phi\to\hat L\Phi$ in $C^1$ as $\eps\to0$.
\end{lemma}
\begin{proof}
We have $\hat L_\eps\Phi = \sum_\omega\Phi\circ g_{\omega,\eps} g_{\omega,\eps}'$.
It suffices to show that for some neighborhood $V$ of $0$,

(i) for each $\omega$, the map $\eps\in V\mapsto \Phi\circ g_{\omega,\eps} g_{\omega,\eps}'\in C^1$ is continuous in $\eps$;

(ii) the series $\sum_\omega \sup_{\eps\in V} \|\Phi\circ g_{\omega,\eps} g_{\omega,\eps}' \|_{C^1}<\infty$.
We skip the proof of (i) since it is similar to (i) in the proof of Lemma  \ref{diff-lemma}. (ii) follows from 
the identity 
\[
(\Phi\circ g_{\omega,\eps} g_{\omega,\eps}')' =
\Phi'\circ g_{\omega,\eps} g_{\omega,\eps}'^2 + \Phi\circ g_{\omega,\eps} g_{\omega,\eps}''
\]
and conditions \eqref{a1} and \eqref{a2}.
\end{proof}
\begin{lemma}\label{lem:G}
We have $G_\eps(q)\to G_0(q)$ in $C^0$ and $G_\eps$ is uniformly bounded in $\L(C_0^1,C^0)$
\end{lemma}
\begin{proof}
We use the fact that the family of operators $\hat L_\eps$ has a uniform spectral gap on $C_0^1$, 
for $\eps$ in a neighborhood of $0$. Hence, these operators are invertible on this space and we have 
$\|(1-\hat L_\eps)^{-1} \|_{C_0^1\to C_0^1}\le C<\infty$.
This proves in particular the second statement. Note that
\[
(G_\eps-G_0)(q) = (I-\hat L_\eps)^{-1} (\hat L_\eps-\hat L)(1-\hat L)^{-1}(q).
\]
By Lemma~\ref{lem:0} with $\Phi=(I-\hat L)^{-1}(q)$ and the previous observations this proves the first statement.
\end{proof}
\subsection{Proof of Theorem~\ref{thm:A}}\label{ptha}
We first prove a  lemma that will be used in the linear response formula in Theorem \ref{thm:A}. 
\begin{lemma}\label{lem:3a}
For any differentiable function $\Phi$, the function $\Phi\circ g_{\omega,\eps} g_{\omega,\eps}'$ is differentiable with respect to $\eps$
and we have on $[0,1]$
\begin{equation}\label{eq:2.0}
\partial_\eps (\Phi\circ g_{\omega,\eps}g_{\omega,\eps}') = \Phi'\circ g_{\omega,\eps} \partial_\eps g_{\omega,\eps}g_{\omega,\eps}' + \Phi\circ g_{\omega,\eps} \partial_\eps g_{\omega,\eps}'.
\end{equation}
\end{lemma}
\begin{proof}
The proof follows by differentiating with respect to $\eps$ and is similar to \eqref{diffeps}.
\end{proof}
\noindent {\bf Strategy of the proof of Theorem \ref{thm:A}.} The argument starts from the first order expansion for $\hat h_\eps$ in $C^0$ 
\[
\hat h_\eps = \hat h + \eps \hat h^* + o(\eps).
\]
Using this, we then obtain, by the second statement of Lemma~\ref{3rdstep} below and relation ~\eqref{eq:density} the following expansion in $\B$
\[
h_\eps =F_\eps(\hat h_\eps) = F_\eps(\hat h) + \eps F_\eps(\hat h^*) + o(\eps).
\]
Finally, we obtain by Lemma~\ref{4thstep} below and the first statement of Lemma~\ref{3rdstep} below the first order expansion of $h_\eps$ in $\B$
\[
h_\eps = h + \eps( Q\hat h + F_0(\hat h^*)) + o(\eps),
\]
which finishes the proof of the theorem. 
\begin{lemma}\label{4thstep}
The map $\eps\mapsto F_{\eps}\hat h$ is differentiable as an element in $\mathcal B$ and $\partial_\eps F_\eps\hat h|_{\eps=0} = Q\hat h$, where $Q$ is defined in \eqref{eq:Q}.
\end{lemma}
\begin{proof}
It suffices to show that

(i) for each $\omega$, the map $\eps\in V\mapsto \hat h\circ g_{\omega,\eps} g_{\omega,\eps}'\in\mathcal B$ is differentiable;

(ii) the series $\sum_\omega \sup_{\eps\in V} \|\partial_\eps(\hat h\circ g_{\omega,\eps} g_{\omega,\eps}')\|_{\mathcal B}<\infty$.\\

We skip the proof of (i) as, by using \eqref{comm}, it follows similar steps as in the proof of (i) in Lemma \ref{diff-lemma}. For (ii), using \eqref{eq:2.0} of Lemma \ref{lem:3a} we have
\begin{equation}
\begin{split}
 \sum_\omega \sup_{\eps\in V} \|\partial_\eps(\hat h\circ g_{\omega,\eps} g_{\omega,\eps}')\|_{\mathcal B}&\le\sum_\omega \sup_{\eps\in V} \|\hat h'\circ g_{\omega,\eps} \cdot\partial_{\eps} g_{\omega,\eps}\cdot g'_{\omega,\eps}\|_{\mathcal B}\\
 &+\sum_\omega \sup_{\eps\in V} \|\hat h\circ g_{\omega,\eps}\cdot\partial_\eps g_{\omega,\eps}'\|_{\mathcal B}\\
 &\le C\sum_\omega \sup_{\eps\in V} \| g_{\omega,\eps}'\|_{\mathcal B}+C\sum_\omega \sup_{\eps\in V} \|\partial_\eps g_{\omega,\eps}'\|_{\mathcal B},
\end{split}
\end{equation}
where we have used the fact that $\hat h$ is $C^1$ and assumptions \eqref{a4} and \eqref{a4.0}. The rest of the proof follows from assumptions \eqref{a4'} and \eqref{a5}. 
\end{proof}
\begin{lemma}\label{3rdstep}
$F_\eps(\hat h^*)\to F_0(\hat h^*)$ in $\mathcal B$ and $F_\eps$ is uniformly bounded in $\L(C^0,\mathcal B)$.
\end{lemma}
\begin{proof} 
To prove uniform boundedness we use assumption \eqref{a4'} to get, for $\Phi \in C^0$,
\begin{equation*}
\begin{split}
||F_\eps(\Phi)||_{\mathcal B}&=||1_\Delta\Phi + (1-1_\Delta)\sum_{\omega\in\Omega}\Phi\circ g_{\omega,\eps}g_{\omega,\eps}'||_{\mathcal B}\\
&\le ||\Phi||_{C^0}+ ||\Phi||_{C^0}\sum_{\omega\in\Omega}\sup_{\eps\in V}||g_{\omega,\eps}'||_{\mathcal B}\le C||\Phi||_{C^0}.
\end{split}
\end{equation*}
Next, the map $g_{\omega,\eps}$ converges to $g_{\omega,0}$ in the $C^1$ norm. Hence for the continuous function $\Phi=\hat h^*\in C^0$ we have
$\Phi \circ g_{\omega,\eps}g_{\omega,\eps}'$ converges uniformly to
$\Phi \circ g_{\omega,0}g_{\omega,0}'$. This together with the normal convergence above shows the continuity of $F_\eps(\hat h^*)\in\mathcal{B}$ at $\eps=0$.
\end{proof}

\section{Verifying the assumptions for Pomeau-Manneville type maps}\label{LSVmaps}
We verify the assumptions of Section \ref{setup} for the family of intermittent maps studied by Liverani-Saussol-Vaienti \cite{LSV} which is a version of the Pomeau-Manneville family \cite{PM}. Let $0<\alpha<\infty$, and define
\begin{equation}\label{LSV-map}
T_{\alpha}(x)=\begin{cases}
       x(1+2^{\alpha}x^{\alpha}) \quad x\in[0,\frac{1}{2}]\\
       2x-1 \quad \quad \quad x\in(\frac{1}{2},1]
       \end{cases}.
\end{equation}
Note that $x=0$ is a neutral fixed point for the map $T_\alpha$ which is consequently a non-uniformly expanding map of the interval (on two pieces). Following Korepanov \cite{K}, we use the following notation
$$
\text{logg}(n)=\begin{cases}
      1\quad\quad\quad n\le e\\
       \log (n) \quad n>e
       \end{cases},
       $$
 and we let $E_{\alpha}:[0,1/2]\to[0,1]$, $E_{\alpha}x=T_{\alpha}x$ be the left branch of $T_\alpha$. Let $z\in[0,1]$, and write $z_{n}:=E^{-n}_{\alpha}(z)$; $z:=z_0$. Let $\hat T_{\omega}:={\hat T_{\alpha}|}_{[\omega]}$ as defined in Section \ref{setup}, then $\hat T_{\omega}(z)=E_{\alpha}^n(T_{\alpha}(z))=E_{\alpha}^n(2z-1)$ for $z\in[\omega]$, and for $z\in [1/2,1]$ $T_{\alpha}(g_{\omega}(z))=2g_{\omega}(z)-1=z_n$. Note that $z_0=z$, $z'_0=1$, $z''_0=z'''_0=0$, for $n\ge 1$ $z_n\le 1/2$, and 
\begin{equation}\label{rec}
z_n=z_{n+1}(1+2^{\alpha}z^\alpha_{n+1});
\end{equation}
\begin{equation}\label{drec}
z_n'=(1+(\alpha+1)2^{\alpha}z^\alpha_{n+1})z'_{n+1}.
\end{equation}
It is well known, see for example \cite{young99}, that $z_n\sim \frac{1}{2 \alpha^{1/\alpha}}n^{-1/\alpha}$. In \cite{K} Korepanov proved
\begin{lemma}\label{korep}
We have
\begin{enumerate}
\item $\frac{C}{n}z_0^{\alpha}\le z^{\alpha}_n\le\frac{C}{n}$, and $-\log(z_n)\le C[\text{logg}(n)-\log z_0]$;
\item 
\begin{equation}\label{list1} 
0\le z'_n\le C (1+nz_0^\alpha\alpha2^{\alpha})^{-1/\alpha-1};
\end{equation}
\item $0\le\frac{z''_n}{z'_n}\le Cz_0^{-2}/\max\{n,1\}$;
\item $\frac{\pa z_n}{z_n}\le C\normalfont\text{logg}(n)[\text{logg}(n)-\log z_0]$ and
\begin{equation}\label{list2} 
\pa z_n\le  C\frac{\normalfont\text{logg}(n)}{n^{1/\alpha}}[ \normalfont\text{logg}(n)-\log z_0];
\end{equation}
\item $|\frac{\pa z'_n}{z'_n}|\le C(\normalfont\text{logg}(n))^2[\text{logg}(n)-\log z_0];$
\item $|\frac{\pa z''_n}{z'_n}|\le Cz_0^{-2}(\normalfont\text{logg}(n))^2[\text{logg}(n)-\log z_0]$.
\end{enumerate}
\end{lemma}
The above list shows that our assumptions \eqref{a1},  \eqref{a2}, \eqref{a3} and \eqref{a4} are satisfied for the LSV family. We still have to show that assumptions \eqref{a4.0}, \eqref{a4'} and \eqref{a5} hold. This will be done in the following lemma. 
\begin{lemma}\label{uniform}
Let\footnote{Here $\alpha_0$ is understood as the parameter corresponding to the unperturbed map; i.e., equivalent to the case $\eps=0$ in Section \ref{setup}.}  $\alpha_0<\gamma$. Let $U$ be a neighbourhood of $\alpha_0$ such that $\gamma\notin U$. We have
\begin{enumerate}
\item $\sum_n\sup_{\alpha\in U}\sup_{z\in(0,1]}|z^\gamma( g^n)'(z)|<C$;
\item $\sup_{n}\sup_{\alpha\in U}\sup_{z\in(0,1]}|\pa g^n(z)|<\infty$;
\item $\sum_n\sup_{\alpha\in U}\sup_{z\in(0,1]}|z^\gamma\pa  (g^n)'(z)|<\infty$.
\end{enumerate}
\end{lemma}
\begin{proof}
For (a), by \eqref{list1}, we have
$$\sum_n\sup_{\alpha\in U}\sup_{z\in[0,1]}|z^\gamma (g^n)'(z)|\le C\sum_n\sup_{\alpha\in U}\sup_{z\in(0,1]}z^\gamma  (1+nz^\alpha\alpha2^{\alpha})^{-1/\alpha-1}<C.$$
For (b), we only discuss the case for $z\in(0,1/2]$. The other case is the same\footnote{In fact Lemma 5.6 of \cite{K} provides an estimate which works only for $z\in(1/2,1]$.}, with a small change in notation. Using \eqref{rec} we have
\begin{equation}\label{first-partial}
\pa z_{j+1}=\frac{\pa z_j+2^{\alpha}z_{j+1}^{\alpha+1}(-\log 2z_{j+1})}{1+(\alpha+1)2^{\alpha}z_{j+1}^{\alpha}}>0.
\end{equation}
Consequently
$$0<\pa z_{j+1}-\pa z_j\le2^{\alpha}z_{j+1}^{\alpha+1}(-\log 2z_{j+1}).$$
Noticing that $\pa z_0=0$, and summing up, we get
\begin{equation}\label{es1}
\pa z_{n+1}\le 2^{\alpha}\sum_{j=1}^{n+1}z_{j}^{\alpha+1} (-\log 2z_{j}).
\end{equation}
Therefore, using \eqref{es1}, we have
\begin{equation}\label{es2}
\begin{split}
&\sup_{n}\sup_{\alpha\in U}\sup_{z\in[0,1/2]}|\pa g^n|\\
&\le C\sup_n\sup_{\alpha\in U}\sum_{j=1}^{n}((j)^{-1/\alpha})^{\alpha+1}(-\log (j)^{-1/\alpha})\\
&\le C\sup_n\sup_{\alpha\in U}\sum_{j=1}^{n}j^{-1-1/\alpha}\log (j)<\infty.
\end{split}
\end{equation}
For (c), using the commutation relation $\partial_\alpha z_{n}' = (\partial_\alpha z_{n})^{'}$, \eqref{first-partial} and \eqref{drec}, we get
$$\frac{\pa z'_j}{z'_j}-\frac{\pa z'_{j+1}}{z'_{j+1}}=\frac{2^{\alpha}z_{j+1}^{\alpha}+(\alpha+1)2^{\alpha}z_{j+1}^{\alpha}\log(2z_{j+1})+\alpha(\alpha+1)2^{\alpha}z_{j+1}^{\alpha-1}\pa z_{j+1}}{1+(\alpha+1)2^{\alpha}z_{j+1}^{\alpha}}.$$
Noticing that $\pa z_0=0$, and summing up, we get
$$-\frac{\pa z'_{n+1}}{z'_{n+1}}=\sum_{j=0}^{n}\frac{2^{\alpha}z_{j+1}^{\alpha}+(\alpha+1)2^{\alpha}z_{j+1}^{\alpha}\log(2z_{j+1})+\alpha(\alpha+1)2^{\alpha}z_{j+1}^{\alpha-1}\pa z_{j+1}}{1+(\alpha+1)2^{\alpha}z_{j+1}^{\alpha}},$$
which is equivalent to
$$-\pa z'_{n+1} =z'_{n+1}\sum_{j=1}^{n+1}\frac{2^{\alpha}z_{j}^{\alpha}+(\alpha+1)2^{\alpha}z_{j}^{\alpha}\log(2z_{j})+\alpha(\alpha+1)2^{\alpha}z_{j}^{\alpha-1}\pa z_{j}}{1+(\alpha+1)2^{\alpha}z_{j}^{\alpha}}.$$
Therefore,
\begin{equation}\label{es3}
\begin{split}
\sum_{n}\sup_{\alpha\in U}\sup_{z\in(0,1/2]}|z^{\gamma}\pa (g^n)'|&\le C\sum_{n=1}^{\infty}\sup_{\alpha\in U}\sup_{z\in(0,1/2]}|z^{\gamma}\cdot z'_{n}|\sum_{j=1}^n z_j^{\alpha}\\
&+C\sum_{n=1}^{\infty}\sup_{\alpha\in U}\sup_{z\in(0,1/2]}|z^{\gamma}\cdot z'_{n}|\sum_{j=1}^n z_j^{\alpha}|\log(z_j)|\\
&+C\sum_{n=1}^{\infty}\sup_{\alpha\in U}\sup_{z\in(0,1/2]}|z^{\gamma}\cdot z'_{n}|\sum_{j=1}^n z_j^{\alpha-1}|\pa z_j|\\
&:=(I)+(II)+(III).
\end{split}
\end{equation}
We use \eqref{list1} to show that $(I)$ and $(II)$ are finite, and \eqref{list1}, \eqref{es1} to show that $(III)$ is finite. Indeed,
$$(I)\le C\sum_{n=1}^{\infty}\sup_{\alpha\in U}\sup_{z\in(0,1/2]}z^{\gamma}\cdot(1+n z^\alpha\alpha2^{\alpha})^{-1/\alpha-1}\sum_{j=1}^n j^{-1}<\infty;$$
$$(II)\le C\sum_{n=1}^{\infty}\sup_{\alpha\in U}\sup_{z\in(0,1/2]} z^{\gamma}\cdot(1+n z^\alpha\alpha2^{\alpha})^{-1/\alpha-1}\sum_{j=1}^n j^{-1}\log(j) <\infty;$$
\begin{equation*}
(III)\le C\sum_{n=1}^{\infty}\sup_{\alpha\in U}\sup_{z\in(0,1/2]}z^{\gamma}\cdot(1+n z^\alpha\alpha2^{\alpha})^{-1/\alpha-1}\sum_{j=1}^nj^{-1}\sum_{k=1}^{j}k^{-1-1/\alpha}\log k<\infty.
\end{equation*}
\end{proof}

\section*{Acknowledgements}
We would like to thank Dalia Terhesiu for useful discussions during ESI workshop ``Thermodynamic Formalism and Mixing'', in particular for encouraging us to incorporate the infinite measure case.

\end{document}